\documentclass{article}
\usepackage{amsfonts}
\usepackage{amsmath,amssymb,amsthm}
\usepackage{bbm}
\usepackage{mathrsfs}
\usepackage{graphics}
\numberwithin{equation}{section}

\usepackage[top=22mm, bottom=22mm, left=25mm, right=25mm]{geometry}

\newtheorem{theorem}{Theorem}[section]
\newtheorem{lemma}[theorem]{Lemma}
\newtheorem{definition}[theorem]{Definition}

\newtheorem{proposition}[theorem]{Proposition}
\newtheorem{corollary}[theorem]{Corollary}

\newcommand{\N}{\mathbb{N}}
\newcommand{\Z}{\mathbb{Z}}

\begin{document}
	\title{{\bf Completely  Li--Yorke chaotic homeomorphisms with positive entropy}}
    \author{Zijie Lin and Xiangdong Ye}
\renewcommand{\thefootnote}{}

    \date{}
	\maketitle

\begin{abstract}
	It is an open problem  whether a homeomorphism on a compact metric space satisfying that each proper pair is either positively or negatively Li--Yorke, called completely Li--Yorke chaotic, can have positive  entropy.
	In the present paper, an affirmative answer to this question is given.

In fact, for each homeomorphism $T$ with positive  entropy such that each proper pair is not two-sided asymptotic, a completely  Li--Yorke chaotic homeomorphism with positive entropy associated with the given homeomorphism can be constructed.
\end{abstract}

\section{Introduction}
A pair $(X,T)$ is called a \emph{topological dynamical system} (t.d.s. for short) if $X$ is a compact metric space with a metric $d_{X}$ and $T:X\to X$ is a homeomorphism. A \emph{proper} pair of $X$ is a pair $(x,y)\in X\times X$ with $x\neq y$. Below we will state our main results and the ideas
of the construction.

\subsection{Main results}

Chaotic phenomenon is one of the most important subjects in the study of dynamical systems.
Li and Yorke \cite{LY75} first introduced the notion of Li--Yorke chaos.
Let $(X,T)$ be a t.d.s. A pair $(x,y)$ of $X$ is called a \emph{positively proximal pair for $T$} if $\liminf_{n\to\infty}d_X(T^nx,T^ny)=0$, and called a \emph{positively asymptotic pair for $T$} if $\lim_{n\to\infty}d_X(T^nx,T^ny)=0$.
A proper pair is called a \emph{positively scrambled pair (or Li--Yorke pair) for $T$} if it is positively proximal pair for $T$ but not a positively asymptotic pair for $T$.
A t.d.s. is called \emph{Li--Yorke chaotic} if there is an uncountable \emph{scrambled set} $S\subset X$, that is, for any two distinct points $x,y\in S$, $(x,y)$ is a positively scrambled pair.

Now we give a short list of works related to the study of scrambled sets.
Li and Yorke \cite{LY75} showed that a map on a closed interval with a periodic point of period three is Li--Yorke chaotic.
Misiurewicz \cite{Mis85} proved that there is a continuous map on a closed interval such that it has a scrambled set with full Lebesgue measure.
Mai \cite{Mai97} constructed a continuous map on the non-compact space $(0,1)^n$ such that the whole space becomes a scrambled set.
Huang and Ye \cite{HY01} showed that there is a t.d.s. such that the whole space becomes a scrambled set, called  \emph{positively completely scrambled}.
Such an example can be given by using the graph covers, see Shimomura \cite{Shi16} or by considering the so-called
almost equicontinuous systems with a unique minimal point, see Glasner and Weiss \cite{GW93}.
Answering some question asked in \cite{HY01}, Boro\`{n}ski, Kupka and Oprocha \cite{BKO19} showed that for any intger $n\ge 1$, there is a topological mixing homeomorphism on an $n$-dimensional continuum, which is positively completely scrambled.

Another question asked by the authors in \cite{HY01} is whether a positively completely scrambled homeomorphism can have positive entropy.
As a by-product the question was answered negatively by Blanchard, Host and Ruette \cite{BHR02} who proved that positive entropy implies the existence of a proper positively asymptotic pair. This result shows that if the whole space is a scrambled set, that is, all the proper pairs are positively scrambled pair, then the map should have zero entropy.

This leads us to consider the natural question for positive Li-Yorke pairs or negative Li-Yorke pairs. There are examples that have positive entropy and also have the following property: there is no proper pair which is both positively asymptotic and negatively asymptotic, see for example the work by Lind and Schmidt \cite{LS99}.

A t.d.s. $(X,T)$ is called \emph{completely scrambled} or \emph{completely Li--Yorke chaotic} if each proper pair is either a positively or negatively scrambled pair. 
This means that some of proper pairs are positively scrambled pairs and the complement of them are negatively scrambled pairs.

Unlike the positively completely scrambled situation, we find out that a completely Li--Yorke chaotic homeomorphism with positive entropy
exists, which gives a positive answer to a problem in \cite[Problem 1]{LiYe}. That is,

\begin{theorem}\label{t:main}
	There is a t.d.s. $(X,T)$ such that it is completely Li--Yorke chaotic and has positive topological entropy.
\end{theorem}

In fact, such an example can be constructed from any given t.d.s.  with a mild condition, i.e. the system has positive entropy and has no proper pair which are both positively and negatively asymptotic.

\medskip

For the system $(X,T)$ in Theorem \ref{t:main}, if we take an ergodic $T$-invariant probability Borel measure $\mu$ on $(X,T)$ with positive measure-theoretic entropy, then the support of $\mu$ is a transitive subsystem of $(X,T)$ with positive topological entropy, which is also completely Li--Yorke chaotic. Thus we have the following corollary.

\begin{corollary}
	There is a transitive t.d.s. such that it is completely  Li--Yorke chaotic and has positive topological entropy.
\end{corollary}

We remark that if we require that each pair in a t.d.s. $(X,T)$ is either positively or negative recurrent, then the entropy should be zero,
see the recent work by Hochman \cite{Hoch}.

\subsection{Main ideas}

Our main idea is to construct a proximal system from a given one such that this proximal system can carry the information of each orbit in the given system. 

Namely, given a t.d.s. $(X,T)$, we will construct a subsystem of $(X_*^{\Z},\sigma)$, where $X_*$ is the space by adding an isolating fixed point $*$ to $X$ and $\sigma$ is the left shift map.
Choosing a suitable subsequence $\{n_k\}_{k=0}^{\infty}$ of $\N$.
For each $n_k$, we divide the integer interval $\{0,1,2,\dots,n_k-1\}$ into two parts: the ``proximal part" and the ``orbit part".
We put the fixed point $*$ into the proximal part, and put the orbit in $X$ into the orbit part.
And then we obtain a closed subset of $X_*^{\Z}$, which relies only on $n_k$ coordinates.
Finally, collecting all those points and their orbit under the left shift map, we construct a proximal system $Y\subset X_*^{\Z}$ depending on $(X,T)$.

The proximal part is required to intersect with its translation, since every point must be proximal to its shift when $(Y,\sigma)$ is a proximal system.
On the other hand, the orbit part is needed to be large enough so that $(Y,\sigma)$ preserves positive topological entropy.
This means that  the proximal part should be small enough.
We find out that such a partition exists for large enough integer $n$ (see Lemma \ref{l:AB}).
So we can show that the constructed system $Y\subset X_*^{\Z}$ has positive topological entropy when so does $(X,T)$, and all the pairs in $Y$ are positively proximal pairs for both $\sigma$ and $\sigma^{-1}$.

Finally, choose a suitable system as $(X,T)$, that is, a homeomorphism with positive topological entropy satisfying that all the proper pairs are positively non-asymptotic pairs for $T$ or $T^{-1}$.
Since the orbit part contains the information of orbits in $X$ and all the proper pairs in $X$ are positively non-asymptotic pairs for  $T$ or $T^{-1}$, we can show that it is same for $(Y,\sigma)$.
Thus the constructed system $(Y,\sigma)$ is completely Li--Yorke chaotic and has positive topological entropy.

\section{Preliminaries}
For subsets $A,B\subset\Z$ and $n\in\Z$, write
\[
A+B:=\{a+b\in\Z:a\in\Z,b\in\Z\},
\]
\[
n+A:=\{n+a\in\Z:a\in A\}\text{ and }nA:=\{na\in\Z:a\in A\}.
\]

For a t.d.s. $(X,T)$ and $x\in X$, denote by $\mathcal{O}_{+}(x,T)$ the closure of the set $\{T^nx:n\ge 0\}$, called the \emph{positive orbit closure for $T$} of $x$.
And denote by $\mathcal{O}_{\Z}(x,T)$ the closure of the set $\{T^nx:n\in\Z\}$, called the \emph{$\Z$-orbit closure for $T$} of $x$.
A point $x\in X$ is called a \emph{fixed point} if $Tx=x$.
A point $x\in X$ is called a \emph{minimal point} if its orbit closure $\mathcal{O}_{+}(x,T)$ does not contain any proper closed $T$-invariant subset. A fixed point is a minimal point.

For a t.d.s. $(X,T)$, we define the product space
\[
X^\Z:=\{(x_i)_{i\in\Z}:x_i\in X,i\in\Z\}
\]
with the product topology, and the shift map $\sigma: X^\Z\to X^\Z$ by $\sigma(x)=(x_{i+1})_{i\in\Z}$ for $x=(x_i)_{i\in\Z}$.
So $(X^\Z,\sigma)$ is a t.d.s. 

For $x=(x_i)_{i\in\Z}\in X^\Z$ and $p<q\in\Z$, write
\[
x|_{[p,q)}=(x_{p},x_{p+1},\dots ,x_{q-1}).
\]
For $z\in X$, we write $z^\Z:=(z)_{i\in\Z}\in X^\Z$ and $z^{N}=(z,z,\dots,z)\in X^N$ for any integer $N\ge 1$.

\subsection{Pairs in topological dynamical systems}

In this subsection, we give some definitions related to a t.d.s. 

\begin{definition}\label{d:pair}
	Let $(X,T)$ be a t.d.s. 
	A pair $(x,y)\in X\times X$ is called
	\begin{itemize}
		\item \emph{positively proximal for $T$} if $\liminf_{n\to+\infty}d_X(T^nx,T^ny)=0$;
		\item \emph{negatively proximal for $T$} if it is positively proximal for $T^{-1}$;
		\item \emph{positively asymptotic for $T$} if $\lim_{n\to+\infty}d_X(T^nx,T^ny)=0$;
		\item \emph{negatively asymptotic for $T$} if it is positively asymptotic for $T^{-1}$;
		\item \emph{{two-sided asymptotic} for $T$} if it is both positively and negatively asymptotic for $T$;
		\item \emph{positively scrambled (or Li--Yorke) pair for $T$} if it is positively proximal and
		$$\limsup_{n\to+\infty}d_X(T^nx,T^ny)>0;$$
		\item \emph{negatively scrambled (or Li--Yorke) pair for $T$} if it is positively scrambled pair for $T^{-1}$.
		\item \emph{ scrambled (or Li--Yorke) pair for $T$} if it is either a positively or negatively scrambled pair for $T$.
	\end{itemize}
\end{definition}

Notice that a pair $(x,y)$ is not two-sided asymptotic if and only if there is some $\delta>0$ such that $\#\{n\in\Z:d_{X}(T^nx,T^ny)>\delta\}=\infty$ where $\#(\cdot)$ stands for the cardinality of a set.

\subsection{Topological entropy}
In this subsection, we give some notations concerning topological entropy.

Let $(X,T)$ be a t.d.s.
For  $\epsilon>0$ and integer $n>0$, a subset $E\subset X$ is called \emph{$(n,\epsilon)$-separated} if for any distinct $x,y\in E$, there is $0\le n'< n$ such that $d_{X}(T^{n'}x,T^{n'}y)>\epsilon$.
Denote by $s_{n}(X,T,\epsilon)$ the greatest number of all $(n,\epsilon)$-separated sets.
The \emph{topological entropy} of $(X,T)$ can be defined by
\[
h_{\rm top}(X,T)=\lim_{\epsilon\to 0}\limsup_{n\to\infty}\frac{\log s_n(X,T,\epsilon)}{n}=\lim_{\epsilon\to 0}\liminf_{n\to\infty}\frac{\log s_n(X,T,\epsilon)}{n}.
\]

\subsection{Completely  Li--Yorke chaotic homeomorphisms}

In this subsection, we introduce some basic properties of completely  Li--Yorke chaotic homeomorphisms.
The following is well known.

\begin{lemma}\label{l:fixedpoint}
	Let $(X,T)$ be a t.d.s. 
	If any pair $(x,y)$ is either positively or negatively proximal for $T$, then there is a unique fixed point $p$ for $T$. 
	Moreover, $p$ is the unique minimal point of $(X,T)$.
\end{lemma}

\begin{proof}
	The proof is similar to Proposition 2.2 in \cite{HY01}.
	Choose any $x\in X$ and then we have $(x,Tx)$ is a positively or negatively proximal pair for $T$.
	If $(x,Tx)$ is a positively proximal pair for $S\in\{T,T^{-1}\}$, then choose a sequence $\{n_k\}$ of positive integers such that $S^{n_k}x\to p$ for some $p\in X$ and $\lim_{k\to\infty}d_X(S^{n_k}x,S^{n_k+1}x)=0$.
	Thus $Tp=p$, that is, $p$ is a fixed point.
	
	If there is another minimal point $x\neq p$ for $(X,T)$, then $p\notin \mathcal{O}_{\Z}(x,T)$.
	So
	\[
	d_X(p,\mathcal{O}_{\Z}(x,T)):=\inf_{z\in \mathcal{O}_{\Z}(x,T)}d_X(p,z)>0.
	\]
	Thus $(x,p)$ is not a positively nor negatively proximal pair for $T$, a contradiction.
\end{proof}

\begin{proposition}
	Let $(X,T)$ be a t.d.s. 
	If any $(x,y)$ is either a positively or negatively proximal pair for $T$,
then all the pairs of $X$ are both positively and negatively proximal pair for $T$.
\end{proposition}

\begin{proof}
	Fix any $x\neq y\in X$.
	Then there is a minimal point $(z,z')\in \mathcal{O}_{+}((x,y),T\times T)$ for $T\times T$.
	By Lemma \ref{l:fixedpoint}, the fixed point $p\in X$ is the unique minimal point.
	Then the fixed point $(p,p)\in X\times X$ is also the unique minimal point of $(X\times X,T\times T)$.
	Thus we have $(z,z')=(p,p)\in \mathcal{O}_{+}((x,y),T\times T)$, and then $(x,y)$ is a positively proximal pair for $T$.
	By a similar argument, $(x,y)$ is also a negatively proximal pair for $T$.
\end{proof}

\section{Induce a proximal system}\label{s:proximal}

In this section, we will construct an induced  proximal system from a given t.d.s. and show the relationship between their  topological entropy.

Let $(X,T)$ be a t.d.s. 
Let $X_*=X\cup\{*\}$, where $*$ is a point which does not belong to $X$.
We can make $X_*$ to be a compact metric space by extending the metric $d_X$ on $X$ to a metric $d_{X_*}$ defined as:
\[
d_{X_*}(*,*)=0,\,d_{X_*}(x,y)=d_X(x,y),\,d_{X_*}(x,*)=\mathrm{diam}(X)+1\text{ for all }x,y\in X
\]
where $\mathrm{diam}(X)$ is the diameter of $X$ under the metric $d_{X}$.
Extend $T$ to a homomorphism $T_*:X_*\to X_*$ by
\[
T_*(*)=*\text{ and }T_*(x)=T(x)\text{ for }x\in X.
\]
So $(X_*,T_*)$ is also a t.d.s.

\medskip

Next, we will construct a proximal subsystem $(Y,\sigma)$ of $(X^\Z_*,\sigma)$ and show the relationship between the topological entropy of $(X,T)$ and $(Y,\sigma)$.
To do so, the following lemma is crucial.

\begin{lemma}\label{l:AB}
	For any $\epsilon>0$, there is a positive integer $N$ such that for any $n>N$, there is a disjoint union $\{0,1,2,\dots,n-1\}=A\cup B$ with $0,n-1\in B$ satisfying the following statements:
	\begin{itemize}
		\item[(i)] $2\le \#A<n\epsilon$;
		\item[(ii)] for any integer $0\le i<n$, there is $a\in A$ such that $a+i\in A\cup (n+A)$;
		\item[(iii)] for any integer $0< i\le \frac{n}{2}+1$, $\#((A+i)\cap B)\ge \sqrt{\frac{n}2}-1$.
	\end{itemize}
\end{lemma}

\begin{proof}
	Fix $\epsilon >0$.
	Choose $N$ such that for any $n>N$,
	\begin{equation}\label{e:n}
	4\le \sqrt{\frac{n}{2}}\text{ and }\sqrt{\frac{2}n}+\frac{3}n<\epsilon.
	\end{equation}
	Fix $n>N$ and set
	\[
	p=\left\lceil\sqrt{\frac{n}{2}}\right\rceil\text{ and }q=\left\lceil\frac{n}{2p}\right\rceil,
	\]
	where $\lceil x\rceil:=\min\{n\in\Z:n\ge x\}$.
	So $n\le2pq$ and
	\begin{equation}\label{e:q}
	\sqrt{n/2}-1<\frac{n}{2(\sqrt{n/2}+1)}<q<\frac{n}{2p}+1\le \sqrt{n/2}+1.
	\end{equation}
	Let $n=k_0p+r_0$ for some integers $k_0>0$ and $0\le r_0<p$.
	So
	\begin{equation}\label{e:n/p}
	2q-3=2(q-1)-1<n/p-1<k_0\le n/p\le 2q.
	\end{equation}
	Let
	\[
	A=\{1,2,\dots,p-1,p\}\cup\{p+r_0,2p+r_0,\dots,qp+r_0,(q+1)p+r_0\}.
	\]
	By the first inequality of (\ref{e:n}) and (\ref{e:q}), we know that $4\le q$.
	And then by (\ref{e:n/p}), we have $k_0>2q-3\ge q+1$.
	So $(q+1)p+r_0<k_0p+r_0=n$.
	Therefore, we have $A\subset\{0,1,2,\dots,n-1\}$ and $n-1\notin A$ by $p>1$.
	The statement (i) holds since
	\[
	\begin{aligned}
	2\le \#A\le p+q+1&< 2\sqrt{\frac{n}2}+3\text{\quad(by (\ref{e:q}))}\\
	&=n\left(\sqrt{\frac{2}n}+\frac{3}n\right)\\
	&<n\epsilon.\text{\quad(by the second inequality of (\ref{e:n}))}
	\end{aligned}
	\]
	
	Let $B=\{0,1,2,\dots,n-1\}\setminus A$.
	So $0,n-1\in B$.

\medskip
	
	To prove (ii), fix $0\le i<n$.
	Let $i=kp+r$ for some integers $k$ and $0\le r<p$.
	So $0\le k\le k_0\le 2q$.
	There are two cases:
	\begin{itemize}
		\item $0\le k\le q$: if $0\le r<r_0$ then
		we have $r_0-r\in A$ and \[i+r_0-r=kp+r_0\in A.\]
		If $r_0\le r<p$, then $p+r_0-r\in A$ and
		\[i+p+r_0-r=(k+1)p+r_0\in A.\]
		\item $q<k\le k_0$: if $r=0$, then we have $0< k_0-k+1\le q$ by $k_0\le 2q$.
		So $(k_0-k+1)p+r_0\in A$ and
		\[(k_0-k+1)p+r_0+i=n+p\in A+n.\]
		
		If $r>0$, then we have $0\le k_0-k<q$ by $k_0\le 2q$.
		So $(k_0-k)p+r_0\in A$ and
		\[(k_0-k)p+r_0+i=n+r\in A+n.\]
	\end{itemize}
	Combining with the above two cases, we prove that (ii) holds.

\medskip
	
	To prove (iii), fix $0<i\le \frac{n}{2}+1$ and divide into two cases.
	
	{\bf Case 1:} if $0<i<p$, then we have $i+jp+r_0\in B$ for all $1\le j\le q$.
	Since $jp+r_0\in A$ for  all $1\le j\le q$, we have
	\[
	\#((A+i)\cap B)\ge q > \sqrt{\frac{n}2}-1
	\]
	where the last inequality holds by (\ref{e:q}).
	
	{\bf Case 2:} if $p\le i\le \frac{n}{2}+1$,
	by the first inequality of (\ref{e:n}), $p<\frac{n}{2}-1$ and then $i+p<n$.
	Since $i+1>p$, then there is at most one element of $i+\{1,2,\dots,p\}$ belongs to $A$.
	Since $\{1,2,\dots,p\}\subset A$, we have
	\[
	\#((A+i)\cap B)\ge p-1 \ge \sqrt{\frac{n}2}-1.
	\]
	
	Combining with the above cases, (iii) is proved.
\end{proof}

\subsection{The construction}

Fix a decreasing sequence $\{\epsilon_k\}_{k=0}^{\infty}$ with $0<\epsilon_k<1/2$ for all $k\ge 0$.
By the above lemma, we inductively define sequences $\{n_k\}_{k=0}^{\infty}$, $\{p_k\}_{k=0}^{\infty}$, and for $k\ge 0$, a disjoint union $\{0,1,2,\dots,n_k-1\}=A_k\cup B_k$ with $0,n_k-1\in B_k$.

For $k=0$, by Lemma \ref{l:AB} for $\epsilon_0$, there are an integer $n_0$ and a disjoint union $\{0,1,2,\dots,n_0-1\}=A_0\cup B_0$ with $0,n_0-1\in B_0$ satisfying (i), (ii) and (iii) in Lemma \ref{l:AB}.
For each integer $k\ge 1$, by Lemma \ref{l:AB} for $\epsilon_k$, there are an integer $n_k>n_{k-1}$ and a disjoint union $\{0,1,2,\dots,n_k-1\}=A_k\cup B_k$ with $0,n_k-1\in B_k$ satisfying (i), (ii) and (iii) in Lemma \ref{l:AB}.
Set $p_{-1}=1$ and $p_{k}=n_kp_{k-1}$ for all $k\ge 0$.

Inductively for $k\ge 0$, assume that we have defined $n_k$, $p_k=n_kp_{k-1}$ (where $p_{-1}=1$) and a disjoint union $\{0,1,2,\dots,n_k-1\}=A_k\cup B_k$ with $0,n_k-1\in B_k$.
Let $b_{-1}=1$ and $b_k=\#B_k$ for $k\ge 0$.
Let $\tau_k:B_k\to \{0,1,\dots,b_k-1\}$ be a strictly increasing bijection.
So for all $k\ge 0$, we have $\tau_k(0)=0$, $\tau_k(n_k-1)=b_k-1$ and
\begin{equation}\label{e:tau_k}
|\tau_k^{-1}(i)-\tau_k^{-1}(j)|\ge |i-j|\text{ for any }i,j\in\{0,1,\dots,b_k-1\}.
\end{equation}

For $k\ge 0$, define $f_k:B_k\to (-(b_k-2),0]\cap\Z$ by
\[
f_k(i)=\left\{
\begin{aligned}
&0,&&i=0,\\
&-(b_k-3),&&i=n_k-1,\\
&\tau_k(i)-(b_k-2),&&otherwise.
\end{aligned}
\right.
\]
More explicitly, see Figure \ref{f:f_k}.
\begin{figure}[h]
	\begin{center}
		\includegraphics{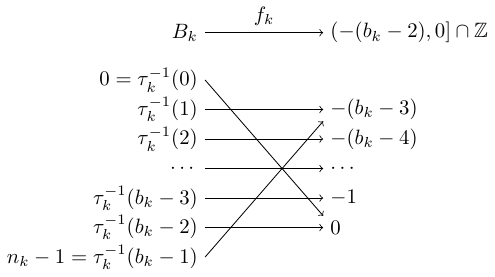}
	\end{center}
	\caption{The map $f_k$}
	\label{f:f_k}
\end{figure}
So
\[
f_k|_{B_k\setminus\{0,n_k-1\}}:B_k\setminus\{0,n_k-1\}\to (-(b_k-2),0]\cap\Z
\]
 is a bijection.
Set $b'_{-1}=1$, $b'_0=\#B_0-2$, $I_0=(-b'_0,0]\cap\Z$ and $I_k=f_k(B_k)\# I_{k-1}+I_{k-1}$ for $k\ge 1$.
So each $I_k=(-b'_k,0]\cap \Z$ is an interval of $\Z$ where $b'_k=b'_{k-1}\#f_k(B_k)$ for $k\ge 0$.

For each $k\ge 0$, define $g_k:B_k\times X_*\to X_*$ by
$g_k(i,x)=T_*^{b'_{k-1}f_k(i)}x$.
If we list the $g_k(i,x)$, $i\in B_k$, then the form is
\[
x,T_*^{-b'_{k-1}(b_k-3)}x,T_*^{-b'_{k-1}(b_k-4)}x,\dots,T_*^{-2b'_{k-1}}x,T_*^{-b'_{k-1}}x,x,T_*^{-b'_{k-1}(b_k-3)}x,
\]
which we add the first point of the orbit $T_*^{-b'_{k-1}(b_k-3)}x,\dots,T_*^{-2b'_{k-1}}x,T_*^{-b'_{k-1}}x,x$ to the end of it, and add the last point of the same orbit to the start of it.
For $k\ge 0$, set
\[
Q_k=\{x=(x_i)_{i\in \Z}\in X^\Z_*:\, x_{ip_{k-1}}=g_k(i,x_{0}),i\in B_k\}
\]
and
\[P_{k}=\{x=(x_i)_{i\in \Z}\in X^\Z_*:x_{ip_{k-1}}=*,\,i\in A_{k}\}.
\]
Then for $k\ge 0$, $P_k$ and $Q_k$ are closed subsets of $X_*^\Z$ since $g_k$ is continuous.
Noticing that for $x=(x_i)_{i\in\Z}\in Q_k$, points $x_{ip_{k-1}}$, $i\in B_k$ belong to the orbit of $x_0$, it is shown that one of points $x_{ip_{k-1}}$, $i\in B_k$ is $*$ if and only if so are all of them.

Let $Y_{-1}=X_*^\Z$, and for $k\ge 0$,
\begin{equation}\label{e:Y_k}
Y_k=Y_{k-1}\cap \bigcap_{n\in\Z}\sigma^{np_k}(P_k\cap Q_k)
\end{equation}
and
\begin{equation}\label{e:Y}
Y=\bigcap_{k=0}^{\infty}\bigcup_{r=0}^{p_k-1}\sigma^{-r}Y_k.
\end{equation}
Since $Y_k$ is closed and $\sigma^{p_k}$-invariant,
$\bigcup_{r=0}^{p_k-1}\sigma^{-r}Y_k$ is closed and $\sigma$-invariant. Then $Y$ is closed and $\sigma$-invariant.
So $(Y,\sigma|_{Y})$ is a subsystem of $(X^\Z_*,\sigma)$.

Notice that $*^\Z\in Y$.
Now, we will give another point $x=(x_i)_{i\in\Z}\in Y$, which shows the form of some points in $Y$.

\begin{lemma}\label{l:sample}
	For any $z\in X$, there is $x=(x_i)_{i\in\Z}\in \bigcap_{k=0}^{\infty}Y_k$ with $x_0=z$.
\end{lemma}

\begin{proof}
	Fix $z\in X$.
	Let $x_{-1}=x_0=z\in X$.
	For $1\le i<n_0=p_0$, set
	\[
	x_{i}=
	\left\{
	\begin{aligned}
	&*,&&i\in A_0,\\
	&g_0(i,x_0),&&i\in B_0.
	\end{aligned}
	\right.
	\]
	Set $x_{-p_0}\in X$ such that $g_0(n_0-1,x_{-p_0})=x_{-1}$.
	For $-n_0\le i<0$, set
	\[
	x_{i}=
	\left\{
	\begin{aligned}
	&*,&&i+n_0\in A_0,\\
	&g_0(i+n_0,x_{-p_0}),&&i+n_0\in B_0.
	\end{aligned}
	\right.
	\]
	Indeed, to ensure $x\in Y_0$, it is shown that for any $n\in\Z$, all of $x_{np_0+i}$, $0\le i<n_0$ are fixed if one of $x_{np_0+i}$, $i\in B_0$ is defined.
	So for any $n\in\Z$, whenever $x_{np_0}$ have been defined, all of $x_{np_0+i}$, $0\le i<p_0$ are defined.
	Then we have defined $x|_{[-p_0,p_0)}$.
	
	Similarly, to ensure $x\in Y_1$, for any $n\in\Z$,  all of $x_{np_1+ip_0}$, $0\le i<n_1$ are fixed if one of $x_{np_1+ip_0}$, $i\in B_1$ is defined.
	So for $1\le i<n_1$, set
	\[
	x_{ip_0}=
	\left\{
	\begin{aligned}
	&*,&&i\in A_1,\\
	&g_1(i,x_0),&&i\in B_1.
	\end{aligned}
	\right.
	\]
	Set $x_{-p_1}$ such that $ g_1(n_1-1,x_{-p_1})=x_{-p_0}$. For $-n_1\le i<0$, set
	\[
	x_{ip_0}=
	\left\{
	\begin{aligned}
	&*,&&i+n_1\in A_1,\\
	&g_1(i+n_1,x_{-p_1}),&&i+n_1\in B_1.
	\end{aligned}
	\right.
	\]
	So by defining $x_{ip_0}$ for $ip_0\in[-p_1,p_1)$, it has defined $x_{ip_0+j}$ for $0\le j<p_0$ and $ip_0\in[-p_1,p_1)$, and then we have defined $x|_{[-p_1,p_1)}$.
	
	Suppose that $x|_{[-p_k,p_k)}$ has been defined for some $k\ge 0$.
	To ensure $x\in Y_{k+1}$, all of $x_{np_{k+1}+ip_k}$, $0\le i<n_{k+1}$ are fixed if one of $x_{np_{k+1}+ip_k}$, $i\in B_{k+1}$ is defined.
	So for $1\le i<n_{k+1}$, set
	\[
	x_{ip_k}=
	\left\{
	\begin{aligned}
	&*,&&i\in A_{k+1},\\
	&g_{k+1}(i,x_{0}),&&i\in B_{k+1}.
	\end{aligned}
	\right.
	\]
	Set $x_{-p_{k+1}}$ such that $g_{k+1}(n_{k+1}-1,x_{-p_{k+1}})=x_{-p_k}$.
	For $-n_{k+1}\le i<0$, set
	\[
	x_{ip_k}=
	\left\{
	\begin{aligned}
	&*,&&i+n_{k+1}\in A_{k+1},\\
	&g_{k+1}(i+n_{k+1},x_{-p_{k+1}}),&&i+n_{k+1}\in B_{k+1}.
	\end{aligned}
	\right.
	\]
	So by defining $x_{ip_k}$ for $ip_k\in[-p_{k+1},p_{k+1})$, it has defined $x_{ip_k+j}$ for $0\le j<p_k$ and $ip_k\in[-p_{k+1},p_{k+1})$, and then we have defined $x|_{[-p_{k+1},p_{k+1})}$.
	
	By the induction process, we have defined $x=(x_i)_{i\in\Z}\in\bigcap_{k=0}^{\infty}Y_k$.
\end{proof}

The following lemma shows the form of points in $Y$.

\begin{lemma}\label{l:r_kofx}
	Let $x=(x_i)_{i\in\Z}\in Y$.
	If $x_m\neq *$ for some $m\in\Z$, then there is a unique sequence $\{r_k\}_{k=0}^{\infty}$ with  $0\le r_k<p_k$ such that $r_{k+1}\equiv r_{k}\mod p_k$, $k \ge 0$ and
	\[
	x\in \bigcap_{k=0}^{\infty}\sigma^{-r_k}Y_k.
	\]
\end{lemma}

\begin{proof}
	Fix $x=(x_i)_{i\in\Z}\in Y$ with $x_m\neq *$ for some $m\in\Z$.
	Then there is a sequence $\{r_k\}_{k=0}^{\infty}$ such that $0\le r_k<p_k$ and $x\in \sigma^{-r_k}Y_k$ for all $k\ge 0$.

	We will prove the uniqueness of $0\le r_k<p_k$ and $r_{k+1}\equiv r_k\mod p_k$ for $k\ge 0$ by induction.
	
	{\bf Step $k=0$:} First, we prove the uniqueness of $r_0$.
	If there are two distinct integers $0\le r_0\neq r'_0<p_0$ such that $x\in \sigma^{-r_0}Y_0$ and $x\in \sigma^{-r'_0}Y_0$, without loss of generality, assume that $r'_0>r_0$.
	Take $m_0\in\Z$ with $r_0+m_0p_0\le m<r_0+(m_0+1)p_0$.
	Since $\sigma^{r_0}x\in Y_0$, we have $m\in r_0+m_0p_0+B_0$ and
	\begin{equation}\label{e:r_0}
	x_i\left\{
	\begin{aligned}
	&=*,&&i\in r_0+m_0p_0+A_0,\\
	&\neq *,&&i\in r_0+m_0p_0+B_0.\\
	\end{aligned}
	\right.
	\end{equation}
	In particular, $x_{r_0+m_0p_0}\neq *$ and $x_{r_0+m_0p_0+p_0-1}\neq *$.
	Since $r'_0+(m_0-1)p_0\le r_0+m_0p_0<r'_0+m_0p_0$ and $r'_0+m_0p_0\le r_0+m_0p_0+p_0-1<r'_0+(m_0+1)p_0$, by $\sigma^{r'_0}x\in Y_0$,
	we have
	\begin{equation}\label{e:r'_0}
	x_i\left\{
	\begin{aligned}
	&=*,&&i\in (r'_0+(m_0-1)p_0+A_0)\cup(r'_0+m_0p_0+A_0),\\
	&\neq *,&&i\in (r'_0+(m_0-1)p_0+B_0)\cup(r'_0+m_0p_0+B_0).\\
	\end{aligned}
	\right.
	\end{equation}
	Depending on whether $r'_0-r_0\le \frac{n_0}{2}$, there are two cases:
	\begin{itemize}
		\item $r'_0-r_0\le \frac{n_0}{2}$:
		By (iii) of Lemma \ref{l:AB}, $(A_0+r'_0-r_0)\cap B_0\neq \emptyset$.
		So
		$$(r'_0+m_0p_0+A_0)\cap (r_0+m_0p_0+B_0)=r_0+m_0p_0+(A_0+r'_0-r_0)\cap B_0\neq \emptyset,$$
		a contradiction to (\ref{e:r_0}) and (\ref{e:r'_0}).
		
		\item $r'_0-r_0>\frac{n_0}{2}$: Then $p_0-(r'_0-r_0)<\frac{n_0}{2}$.
		By (iii) of Lemma \ref{l:AB}, $(A_0+p_0-(r'_0-r_0))\cap B_0\neq \emptyset$.
		So
		$$
		\begin{aligned}
		&(r_0+m_0p_0+A_0)\cap (r'_0+(m_0-1)p_0+B_0)\\
		=&r'_0+(m_0-1)p_0+(A_0+p_0-(r'_0-r_0))\cap B_0\neq \emptyset,
		\end{aligned}
		$$
		a contradiction to (\ref{e:r_0}) and (\ref{e:r'_0}).
	\end{itemize}
	So we can conclude that there is a unique $0\le r_0<p_0$ such that $x\in\sigma^{-r_0}Y_0$.
	
	{\bf Step $k=1$:} Next, we show that $r_1\equiv r_0\mod p_0$ and prove the uniqueness of $r_1$.
	Since $x\in \sigma^{-r_1}Y_1\subset \sigma^{-r_1}Y_0$, by the uniqueness of $r_0$, we have $r_1\equiv r_0\mod p_0$.
	
	If there are two distinct $0\le r_1<r'_1<p_1$ such that $x\in\sigma^{-r_1}Y_1$ and $x\in\sigma^{-r'_1}Y_1$, then $r_1\equiv r'_1\equiv r_0\mod p_0$.
	Take $m_1\in\Z$ with $r_1+m_1p_1\le m<r_1+(m_1+1)p_1$.
	And by $r_0+m_0p_0\le m< r_0+(m_0+1)p_0$ and $r_1\equiv r_0\mod p_0$, we have $[r_0+m_0p_0,r_0+(m_0+1)p_0)\subset[r_1+m_1p_1,r_1+(m_1+1)p_1)$ and $r_0+m_0p_0\in r_1+m_1p_1+p_0\{0,1,2,\dots,n_1-1\}$.
	\begin{figure}[h]
		\begin{center}
			\includegraphics{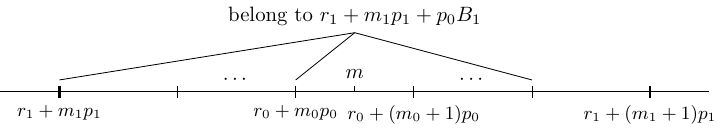}
		\end{center}
	\end{figure}

	By $x_{r_0+m_0p_0}\neq *$, we have $r_0+m_0p_0\in r_1+m_1p_1+p_0B_1$ and
	\begin{equation}\label{e:r_1}
	x_i\left\{
	\begin{aligned}
	&=*,&&i\in r_1+m_1p_1+p_0A_1,\\
	&\neq *,&&i\in r_1+m_1p_1+p_0B_1.\\
	\end{aligned}
	\right.
	\end{equation}
	In particular, $x_{r_1+m_1p_1}\neq *$ and $x_{r_1+(m_1+1)p_1-p_0}\neq *$.
	Since $r'_1+(m_1-1)p_1\le r_1+m_1p_1<r'_1+m_1p_1$, $r'_1+m_1p_1\le r_1+(m_1+1)p_1-p_0<r'_1+(m_1+1)p_1$ and $r_1\equiv r'_1\equiv r_0\mod p_0$, we have $r_1+m_1p_1\in r'_1+(m_1-1)p_1+p_0\{0,1,2,\dots,n_1-1\}$ and $r_1+(m_1+1)p_1-p_0\in r'_1+m_1p_1+p_0\{0,1,2,\dots,n_1-1\}$.
	So we have
	\begin{equation}\label{e:r'_1}
	x_i\left\{
	\begin{aligned}
	&=*,&&i\in (r'_1+(m_1-1)p_1+p_0A_1)\cup(r'_1+m_1p_1+p_0A_1),\\
	&\neq *,&&i\in (r'_1+(m_1-1)p_1+p_0B_1)\cup(r'_1+m_1p_1+p_0B_1).\\
	\end{aligned}
	\right.
	\end{equation}
	Depending on whether $i_1:=\frac{r'_1-r_1}{p_0}\le \frac{n_1}{2}$, there are two cases:
	\begin{itemize}
		\item $i_1\le \frac{n_1}{2}$:
		By (iii) of Lemma \ref{l:AB}, $(A_1+i_1)\cap B_1\neq \emptyset$.
		So
		$$(r'_1+m_1p_1+p_0A_1)\cap (r_1+m_1p_1+p_0B_1)=r_1+m_1p_1+p_0((A_1+i_1)\cap B_1)\neq \emptyset,$$
		a contradiction to (\ref{e:r_1}) and (\ref{e:r'_1}).
		
		\item $i_1>\frac{n_1}{2}$: Then $n_1-i_1<\frac{n_1}{2}$.
		By (iii) of Lemma \ref{l:AB}, $(A_1+n_1-i_1)\cap B_1\neq \emptyset$.
		So
		\[
		\begin{aligned}
		&(r_1+m_1p_1+p_0A_1)\cap (r'_1+(m_1-1)p_1+p_0B_1)\\
		=&r'_1+(m_1-1)p_1+p_0((A_1+n_1-i_1)\cap B_1)\neq \emptyset,
		\end{aligned}
		\]
		a contradiction to (\ref{e:r_1}) and (\ref{e:r'_1}).
	\end{itemize}
	So we can conclude that there is a unique $0\le r_1<p_1$ such that $x\in\sigma^{-r_1}Y_1$.
	
	{\bf Step $k+1$:} Suppose that for some $k\ge 1$, we have proved that for any $1\le K\le k$, $0\le r_K<p_K$ is unique, $r_K\equiv r_{K-1}\mod p_{K-1}$, $r_{K}+m_{K}p_{K}\le m<r_{K}+(m_{K}+1)p_{K}$ and $x_{r_{K}+m_{K}p_{K}}\neq *$.
	We will show that $r_{k+1}\equiv r_k\mod p_k$ and $r_{k+1}$ is unique.
	Since $x\in \sigma^{-r_{k+1}}Y_{k+1}\subset \sigma^{-r_{k+1}}Y_k$, by the uniqueness of $r_k$, we have $r_{k+1}\equiv r_k \mod p_k$.
	
	If there are two distinct $0\le r_{k+1}<r'_{k+1}<p_{k+1}$ such that $x\in \sigma^{-r_{k+1}}Y_{k+1}$ and $x\in \sigma^{-r'_{k+1}}Y_{k+1}$, then $r_{k+1}\equiv r'_{k+1}\equiv r_{k}\mod p_{k}$.
	Take $m_{k+1}\in\Z$ such that $r_{k+1}+m_{k+1}p_{k+1}\le m<r_{k+1}+(m_{k+1}+1)p_{k+1}$.
	And by $r_{k}+m_{k}p_{k}\le m<r_{k}+(m_{k}+1)p_{k}$ and $r_{k+1}\equiv r_{k}\mod p_k$, we have
	$$[r_{k}+m_{k}p_{k},r_{k}+(m_{k}+1)p_{k})\subset [r_{k+1}+m_{k+1}p_{k+1},r_{k+1}+(m_{k+1}+1)p_{k+1})$$
	and $r_{k}+m_{k}p_{k}\in r_{k+1}+m_{k+1}p_{k+1}+p_k\{0,1,2,\dots,n_{k+1}-1\}$.
	By $x_{r_{k}+m_{k}p_{k}}\neq *$, we have $r_{k}+m_{k}p_{k}\in r_{k+1}+m_{k+1}p_{k+1}+p_kB_{k+1}$ and
	\begin{equation}\label{e:r_{k+1}}
	x_i\left\{
	\begin{aligned}
	&=*,&&i\in r_{k+1}+m_{k+1}p_{k+1}+p_kA_{k+1},\\
	&\neq *,&&i\in r_{k+1}+m_{k+1}p_{k+1}+p_kB_{k+1}.\\
	\end{aligned}
	\right.
	\end{equation}
	In particular, $x_{r_{k+1}+m_{k+1}p_{k+1}}\neq *$ and $x_{r_{k+1}+(m_{k+1}+1)p_{k+1}-p_k}\neq *$.
	Since $$r'_{k+1}+(m_{k+1}-1)p_{k+1}\le r_{k+1}+m_{k+1}p_{k+1}<r'_{k+1}+m_{k+1}p_{k+1},$$
	$$r'_{k+1}+m_{k+1}p_{k+1}\le r_{k+1}+(m_{k+1}+1)p_{k+1}-p_{k}<r'_{k+1}+(m_{k+1}+1)p_{k+1}$$
	and $r_{k+1}\equiv r'_{k+1}\equiv r_{k}\mod p_{k}$,
	we have $$r_{k+1}+m_{k+1}p_{k+1}\in r'_{k+1}+(m_{k+1}-1)p_{k+1}+p_{k}\{0,1,2,\dots,n_{k+1}-1\}$$
	and $$r_{k+1}+(m_{k+1}+1)p_{k+1}-p_{k}\in r'_{k+1}+m_{k+1}p_{k+1}+p_{k}\{0,1,2,\dots,n_{k+1}-1\}.$$
	So we have
	\begin{equation}\label{e:r'_{k+1}}
	x_i\left\{
	\begin{aligned}
	&=*,&&i\in (r'_{k+1}+(m_{k+1}-1)p_{k+1}+p_{k}A_{k+1})\cup(r'_{k+1}+m_{k+1}p_{k+1}+p_{k}A_{k+1}),\\
	&\neq *,&&i\in (r'_{k+1}+(m_{k+1}-1)p_{k+1}+p_{k}B_{k+1})\cup(r'_{k+1}+m_{k+1}p_{k+1}+p_{k}B_{k+1}).\\
	\end{aligned}
	\right.
	\end{equation}
	Depending on whether  $i_{k+1}:=\frac{r'_{k+1}-r_{k+1}}{p_{k}}\le \frac{n_{k+1}}{2}$, there are two cases:
	\begin{itemize}
		\item $i_{k+1}\le \frac{n_{k+1}}{2}$:
		By (iii) of Lemma \ref{l:AB}, $(A_{k+1}+i_{k+1})\cap B_{k+1}\neq \emptyset$.
		So
		$$
		\begin{aligned}
		&(r'_{k+1}+m_{k+1}p_{k+1}+p_{k}A_{k+1})\cap (r_{k+1}+m_{k+1}p_{k+1}+p_{k}B_{k+1})\\
		=&r_{k+1}+m_{k+1}p_{k+1}+p_{k}((A_{k+1}+i_{k+1})\cap B_{k+1})\neq \emptyset,
		\end{aligned}
		$$
		a contradiction to (\ref{e:r_{k+1}}) and (\ref{e:r'_{k+1}}).
		
		\item $i_{k+1}>\frac{n_{k+1}}{2}$: Then $n_{k+1}-i_{k+1}<\frac{n_{k+1}}{2}$.
		By (iii) of Lemma \ref{l:AB}, $(A_{k+1}+n_{k+1}-i_{k+1})\cap B_{k+1}\neq \emptyset$.
		So
		$$
		\begin{aligned}
		&(r_{k+1}+m_{k+1}p_{k+1}+p_{k}A_{k+1})\cap (r'_{k+1}+(m_{k+1}-1)p_{k+1}+p_{k}B_{k+1})\\
		=&r'_{k+1}+(m_{k+1}-1)p_{k+1}+p_{k}((A_{k+1}+n_{k+1}-i_{k+1})\cap B_{k+1})\neq \emptyset,
		\end{aligned}
		$$
		a contradiction to (\ref{e:r_{k+1}}) and (\ref{e:r'_{k+1}}).
	\end{itemize}

	Now, we have proved that the sequence $\{r_k\}$ is unique for $0\le r_k<p_k$, $k\ge 0$, and $r_{k+1}\equiv r_{k}\mod p_k$ for all $k\ge 0$.
\end{proof}

\newcommand{\se}[3]{\{#1_#2\}_{#2=#3}^{\infty}}

By Lemma \ref{l:r_kofx}, for any $x\in Y$ with $x_m\neq *$ for some $m\in\Z$, we can always assume that $x\in \bigcap_{k=0}^{\infty}\sigma^{-r_k}Y_k$ for some sequence
$\se rk0$ with $0\le r_k<p_k$ and $r_{k+1}\equiv r_k\mod p_k$ for all $k\ge 0$.

\subsection{Proximality of $Y$}

In this subsection, we show that all pairs of $Y$ are both positively and negatively proximal.
First, we need the following lemma.
\begin{lemma}\label{l:forproximal}
	Let $x=(x_i)_{i\in\Z}\in \bigcap_{k=0}^{\infty}\sigma^{-r_k}Y_k$ for some sequence $\se rk0$ with $0\le r_k<p_k$ and $r_{k+1}\equiv r_k\mod p_k$ for all $k\ge 0$.
	Then for any $k\ge 0$,
	\begin{itemize}
		\item[(i)] $\{i\in\Z:x_i=*\}\supset r_k+p_k\Z+p_{k-1}A_k$;
		\item[(ii)] if
		$x_{r_k+np_k}=*$ for some $n\in\Z$, then
		$x_{r_k+np_k+i}=*$, $i=0,1,2,\dots,p_k-1$.
	\end{itemize}
	Moreover, we have
	\begin{equation}\label{e:proximal}
	\{i\in\Z:x_i=*\}\supset \bigcup_{k=0}^{\infty}(r_k+p_k\Z+p_{k-1}A_k+[0,p_{k-1})\cap\Z).
	\end{equation}
\end{lemma}

\begin{proof}
	We prove by induction.
	For $k=0$, we have $\sigma^{r_0}x\in Y_0\subset \bigcap_{n\in\Z}\sigma^{np_0}P_0$.
	So for any $n\in \Z$, then $x_{r_0+np_0+i}=*$ for all $i\in A_0$.
	Therefore, we have
	$$\{i\in \Z:x_i=*\}\supset r_0+p_0\Z+A_0.$$
	Moreover, by $\sigma^{r_0}x\in Y_0\subset \bigcap_{n\in\Z}\sigma^{np_0}(P_0\cap Q_0)$, if
	$x_{r_0+np_0}=*$ for some $n\in\Z$, then
	$x_{r_0+np_0+i}=*$, $i=0,1,2,\dots,p_0-1$.
	So we have (i) and (ii) hold for $k=0$.
	
	Suppose that (i) and (ii) hold for some $k\ge 0$.
	Since $\sigma^{r_{k+1}}x\in Y_{k+1}\subset \bigcap_{n\in\Z}\sigma^{np_{k+1}}P_{k+1}$ for any $n\in \Z$,
	then $x_{r_{k+1}+np_{k+1}+ip_{k}}=*$ for all $i\in A_{k+1}$.
	Since (ii) holds for $k$ and $r_{k+1}+np_{k+1}+ip_{k}\in r_{k}+p_{k}\Z$ by $r_{k+1}\equiv r_{k}\mod p_{k}$, we have $x_{r_{k+1}+np_{k+1}+ip_{k}+j}=*$ for all $i\in A_{k+1}$ and $j=0,1,2,\dots,p_{k}-1$.
	Therefore, we have
	$$\{i\in \Z:x_i=*\}\supset r_{k+1}+p_{k+1}\Z+p_{k}A_{k+1}+[0,p_{k})\cap\Z.$$
	Moreover, by $\sigma^{r_{k+1}}x\in Y_{k+1}\subset \bigcap_{n\in\Z}\sigma^{np_{k+1}}(P_{k+1}\cap Q_{k+1})$, if $x_{r_{k+1}+np_{k+1}}=*$ for some $n\in\Z$,
	then $x_{r_{k+1}+np_{k+1}+ip_{k}}=*$ for all $i=0,1,2,\dots,n_{k}-1$.
	So (i) and (ii) hold for $k+1$.
	
	Combining with (i) and (ii) for all $k\ge 0$, (\ref{e:proximal}) is proved.
\end{proof}

In Lemma \ref{l:forproximal}, by (\ref{e:proximal}), we also have
\begin{equation}\label{e:non*}
\{i\in\Z:x_i\neq *\}\subset \bigcap_{k=0}^{\infty}(r_k+p_k\Z+p_{k-1}B_k+[0,p_{k-1})\cap\Z).
\end{equation}

\begin{proposition}\label{p:proximal}
	For any $x,y\in Y$, $(x,y)$ is both positively and negatively proximal.
\end{proposition}

\begin{proof}
	Fix $x=(x_i)_{i\in\Z},y=(y_i)_{i\in\Z}\in Y$.
	By Lemma \ref{l:r_kofx}, there are sequences $\se rk0$ and $\se{r'}k0$ such that
	\[
		x\in\bigcap_{k=0}^{\infty}\sigma^{-r_k}Y_k\text{ and }y\in\bigcap_{k=0}^{\infty}\sigma^{-r'_k}Y_k.
	\]
	If $r_k=r'_k$ for all $k\ge 0$, then by (\ref{e:proximal}), $(x,y)$ is both positively and negatively proximal.
	
	If there is some $k_0\ge 0$ such that $r_{k_0}\neq r'_{k_0}$.
	Then by $r_{k+1}\equiv r_k\mod p_k$ and $r'_{k+1}\equiv r'_k\mod p_k$ for all $k\ge 0$, we have $r_{k}\neq r'_{k}$ for all $k\ge k_0$.
	
	Next, we will show that for any $N\ge 0$ there are $n^{+}>N$ and $n^{-}<-N$ such that
	\begin{equation}\label{e:pro}
	x|_{[n^{-}-N,n^{-})}=y|_{[n^{-}-N,n^{-})}=x|_{[n^{+},n^{+}+N)}=y|_{[n^{+},n^{+}+N)}=*^N,
	\end{equation}
	which implies that $(x,y)$ is both positively and negatively proximal.
	
	Fix $N\ge 0$.
	Choose $k>k_0$ with $p_{k}>2N$.
	Without loss of generality, assme that $0\le r_{k+1}<r'_{k+1}<p_{k+1}$.
	Set $r'_{k+1}-r_{k+1}=ap_{k}+b$ for some integers $0\le a<n_{k+1}$ and $0\le b<p_k$.
	We divide $b$ into two cases.
	
	{\bf Case 1:} if $0\le b<p_k/2$, then by (ii) of Lemma \ref{l:AB}, there is $a_1\in A_{k+1}$ such that
	$a_1+a\in A_{k+1}\cup (A_{k+1}+n_{k+1})$.
	So by (\ref{e:proximal}),
	\[
	\begin{aligned}
	\{i\in\Z:y_i=*\}\supset
	&r'_{k+1}+p_{k+1}\Z+a_1p_k+[0,p_k)\cap\Z\\
	=&r_{k+1}+p_{k+1}\Z+(a_1+a)p_k+b+[0,p_k)\cap\Z\\
	\supset & r_{k+1}+p_{k+1}\Z+(a_1+a)p_k+[b,b+N)\cap\Z
	\end{aligned}
	\]
	where the last statement holds by $N<\frac{p_k}2$.
	Since $a_1+a\in A_{k+1}\cup (A_{k+1}+n_{k+1})$ and $b+N<p_k$, we have
	\[
	\begin{aligned}
	\{i\in\Z:x_i=*\}\supset
	&r_{k+1}+p_{k+1}\Z+p_kA_{k+1}+[0,p_k)\cap\Z\quad\text{(by (\ref{e:proximal}))}\\
	\supset &r_{k+1}+p_{k+1}\Z+(a_1+a)p_k+[b,b+N)\cap\Z
	\end{aligned}
	\]
	where the last statement holds by $b<\frac{p_k}2$ and $N<\frac{p_k}2$.
	So
	\[
		\{i\in\Z:x_i=y_i=*\}\supset r_{k+1}+p_{k+1}\Z+(a_1+a)p_k+[b,b+N)\cap\Z,
	\]
	which implies (\ref{e:pro}).
	
	{\bf Case 2:} if $p_k/2\le b<p_k$, then $0<p_k-b<p_k/2$.
	By (ii) of Lemma \ref{l:AB}, there is $a_2\in A_{k+1}$ such that $a_2+(n_{k+1}-1-a)\in A_{k+1}\cup (A_{k+1}+n_{k+1})$.
	So by (\ref{e:proximal}),
	\[
	\begin{aligned}
	\{i\in\Z:x_i=*\}\supset
	&r_{k+1}+p_{k+1}\Z+a_2p_k+[0,p_k)\cap\Z\\
	=&r'_{k+1}+p_{k+1}\Z+(a_2+n_{k+1}-1-a)p_k+(p_k-b)+[0,p_k)\cap\Z\\
	\supset & r'_{k+1}+p_{k+1}\Z+(a_2+n_{k+1}-1-a)p_k+[p_k-b,p_k-b+N)\cap\Z
	\end{aligned}
	\]
	where the last statement holds by $N<\frac{p_k}{2}$.
	Since $p_k-b+N<p_k$ and $a_2+n_{k+1}-1-a\in A_{k+1}\cup (A_{k+1}+n_{k+1})$, we have
	\[
	\begin{aligned}
	\{i\in\Z:y_i=*\}\supset
	&r'_{k+1}+p_{k+1}\Z+p_kA_{k+1}+[0,p_k)\cap\Z\quad\text{(by (\ref{e:proximal}))}\\
	\supset &r'_{k+1}+p_{k+1}\Z+(a_2+n_{k+1}-1-a)p_k+[p_k-b,p_k-b+N)\cap\Z
	\end{aligned}
	\]
	where the last statement holds by $p_k-b<\frac{p_k}{2}$ and $N<\frac{p_k}{2}$.
	So
	\[
	\{i\in\Z:x_i=y_i=*\}\supset r'_{k+1}+p_{k+1}\Z+(a_2+n_{k+1}-1-a)p_k+[p_k-b,p_k-b+N)\cap\Z,
	\]
	which implies (\ref{e:pro}).
\end{proof}

\subsection{Entropy of $Y$}

In this subsection, we show the relationship between the entropy of $X$ and $Y$.
In the proof of Lemma \ref{l:sample}, we can see that $x$ contains the orbit of $x_0$.
Indeed, this also holds for any $x\in Y_k$.
For $k\ge 0$, let $C_k=B_k\setminus\{0,n_k-1\}$.
So $\#C_k=b_k-2$, and $f_k|_{C_k}:C_k\to(-(b_k-2),0]\cap\Z$ is a bijection.
For convenience, we write $f^{-1}_k:=(f_k|_{C_k})^{-1}:(-(b_k-2),0]\cap\Z\to C_k$ for all $k\ge 0$.

For $K\ge 0$, set
\begin{equation}\label{e:J_k}
J_{K}=\Z\cap[0,p_{K})\cap\bigcap_{k=0}^{K}(p_{k}\Z+p_{k-1}C_{k}+[0,p_{k-1}))
.
\end{equation}
So $J_0=C_0$ and for $K\ge 0$,
\[
\begin{aligned}
J_{K+1}=&\Z\cap[0,p_{K+1})\cap\bigcap_{k=0}^{K+1}(p_{k}\Z+p_{k-1}C_{k}+[0,p_{k-1}))\\
=&\Z\cap(p_{K}C_{K+1}+[0,p_{K}))\cap\bigcap_{k=0}^{K}(p_{k}\Z+p_{k-1}C_{k}+[0,p_{k-1}))\\
=&\Z\cap\left(\bigcup_{i\in C_{K+1}}(ip_{K}+[0,p_{K}))\right)\cap\bigcap_{k=0}^{K}(p_{k}\Z+p_{k-1}C_{k}+[0,p_{k-1}))\\
=&\bigcup_{i\in C_{K+1}}\left(ip_{K}+\Z\cap[0,p_{K})\cap\bigcap_{k=0}^{K}(-ip_{K}+p_{k}\Z+p_{k-1}C_{k}+[0,p_{k-1}))\right)\\
=&p_{K}C_{K+1}+\Z\cap[0,p_{K})\cap\bigcap_{k=0}^{K}(p_{k}\Z+p_{k-1}C_{k}+[0,p_{k-1}))
\end{aligned}
\]
where the last equality holds by $p_{K}\in \bigcap_{k=0}^{K}p_{k}\Z$.
So $J_{k+1}=p_{k}C_{k+1}+J_{k}$ for all $k\ge 0$.
Recall that $I_k=(-b'_{k},0]\cap \Z$ and $I_{k+1}=f_{k+1}(B_{k+1})b'_{k}+I_k$  for $k\ge 0$.

\begin{lemma}\label{l:phi_k}
	There are strictly increasing bijections
	$\phi_k: I_k\to J_k$, $k\ge 0$
	such that for all $k\ge 0$,
	\begin{itemize}
		\item[(i)] for any $x\in Y_k$, $x_{\phi_k(n)}=T_*^nx_0$, $n\in I_k$, and
		\item[(ii)] for any $j\in f_{k+1}(B_{k+1})$,  $\phi_{k+1}(jb'_k+I_k)=f^{-1}_{k+1}(j)p_{k}+J_k$.
	\end{itemize}
	
\end{lemma}

\begin{proof}
	Recall that  $f^{-1}_k:=(f_k|_{C_k})^{-1}$ for all $k\ge 0$.
	We prove by induction.
	For $k=0$, define a map $\phi_0: I_0\to J_0=C_0$ by $\phi_0(n)=f_0^{-1}(n)$ for $n\in I_0$.
	Then for any $x\in Y_0$, by $x\in Y_0\subset Q_0$, we have $x_{\phi_0(n)}=g_0(f_0^{-1}(n),x_0)=T_*^nx_0$ for all $n\in I_0$.
	Since $f_0^{-1}$ is a strictly increasing bijection, so is $\phi_0$.
	
	Suppose that we have defined $\phi_k$ as required for some $k\ge 0$.
	Since $\#I_k=b'_{k}$, then the translations $jb'_{k}+I_k$, $j\in\Z$ of $I_k$ are disjoint.
	So we define a map $\phi_{k+1}: I_{k+1}\to J_{k+1}$ by
	\[
	\phi_{k+1}(n)=f_{k+1}^{-1}(j)p_k+\phi_k(n-jb'_{k}),\,n\in jb'_{k}+I_k,\,j\in f_{k+1}(B_{k+1}).
	\]
	Since $\phi_k(I_k)=J_k\subset [0,p_{k})$, $\phi_k$ and $f^{-1}_{k+1}$ are strictly increasing bijections, and $J_{k+1}=p_kC_{k+1}+J_k$, we have $\phi_{k+1}$ is also a strictly increasing bijection and (ii) holds.
	For any $x\in Y_{k+1}$, by $x\in Y_{k+1}\subset Y_k$, we have $\sigma^{f_{k+1}^{-1}(j)p_k}x\in Y_k$ for all $j\in f_{k+1}(B_{k+1})$.
	So for all $j\in f_{k+1}(B_{k+1})$, by the induction hypothesis,	
	\begin{equation}\label{e:phi_1.1}
	x_{f_{k+1}^{-1}(j)p_k+\phi_k(n)}=T_*^nx_{f_{k+1}^{-1}(j)p_k}
	\end{equation}
	for all $n\in I_k$.
	By $x\in Y_{k+1}\subset Q_{k+1}$,
	\begin{equation}\label{e:phi_1.2}
	x_{f_{k+1}^{-1}(j)p_k}=g_{k+1}(f_{k+1}^{-1}(j),x_0)=T_*^{jb'_{k}}x_0
	\end{equation}
	for all $j\in f_{k+1}(B_{k+1})$.
	So by (\ref{e:phi_1.1}) and (\ref{e:phi_1.2}), we have $x_{\phi_{k+1}(n)}=T^nx_0$ for all $n\in I_{k+1}$.
	Therefore, it ends the proof.
\end{proof}

By the relationship between the orbits in $X$ and points in $Y$, we are able to estimate the entropy of $Y$.

\begin{proposition}\label{p:entropy}
	We have
	\[
	h_{\rm top}(Y,\sigma)\ge h_{\rm top}(X,T)\cdot \prod_{k=0}^{\infty}(1-2\epsilon_k).
	\]
\end{proposition}

\begin{proof}
	Since $J_k\subset [0,p_k)$ for all $k\ge 0$, by (i) of Lemma \ref{l:phi_k}, the key is to estimate $\frac{\#I_k}{p_k}=\frac{\#J_k}{p_k}$.
	We prove the following inequality by induction:
	\begin{equation}\label{e:J_K}
	\#J_{K}\ge p_{K}\prod_{k=0}^{K}(1-2\epsilon_k),\,K\ge 0.
	\end{equation}
	For $K\ge 0$, $\#J_0=\#C_0=p_0-\#A_0-2>p_0(1-2\epsilon_0)$ by (i) of Lemma \ref{l:AB}.
	Suppose that (\ref{e:J_K}) holds for some $K\ge 0$.
	Recall that $J_{K+1}=p_{K}C_{K+1}+J_K$, then we have
	\[
	\begin{aligned}
	\#J_{K+1}=&\#C_{K+1}\cdot \#J_K\\
	\ge &(n_{K+1}-\#A_{K+1}-2)\cdot p_{K}\prod_{k=0}^{K}(1-2\epsilon_k)\\
	\ge &n_{K+1}(1-2\epsilon_{K+1})\cdot p_{K}\prod_{k=0}^{K}(1-2\epsilon_k)\quad\text{(by (i) of Lemma \ref{l:AB})}\\
	=&p_{K+1}\prod_{k=0}^{K+1}(1-2\epsilon_k)\quad\text{(by $p_{K+1}=n_{K+1}p_K$).}
	\end{aligned}
	\]
	So (\ref{e:J_K}) is proved.
	For any $\eta>0$, there is $\delta>0$ such that $d_{Y}(x,y)>\delta$ whenever $d_{X_*}(x_0,y_0)>\eta$ for any $x=(x_i)_{i\in\Z},y=(y_i)_{i\in\Z}\in Y$.
	By Lemma \ref{l:sample} and (i) of Lemma \ref{l:phi_k}, we have $s_{p_K}(Y,\sigma,\delta)\ge s_{\#I_K}(X,T,\eta)$ for all $K\ge 0$.
	Since $\#I_K=\#J_K$, we have
	\[
	\frac{\log s_{p_K}(Y,\sigma,\delta)}{p_K}>\frac{\log s_{p_K}(Y,\sigma,\delta)}{\#I_K}\cdot \prod_{k=0}^{K}(1-2\epsilon_k)\ge \frac{\log s_{\#I_K}(X,T,\eta)}{\#I_K}\cdot \prod_{k=0}^{K}(1-2\epsilon_k).
	\]
	Let $K\to\infty$, and by the arbitrariness of $\eta$, it ends the proof.
\end{proof}

By the above lemma, if $(X,T)$ has positive entropy, so does $(Y,\sigma)$ when we take $\se{\epsilon}k0$ satisfying
\begin{equation}\label{e:eps}
\prod_{k=0}^{\infty}(1-2\epsilon_k)>0.
\end{equation}

\section{Completely Li--Yorke chaotic systems with positive entropy}

In this section, we will give a completely  Li--Yorke chaotic system by the constructed system in Section \ref{s:proximal}.
We mainly show that each proper pair of $Y$ is not {two-sided asymptotic} for $\sigma$ where $Y$ is induced as in Section \ref{s:proximal} by some system $(X,T)$.
We divide the proper pairs of $Y$ into three cases, where the first two cases do not have any requirement for $X$, but the last one case does.

First, we show two typical cases for $Y$ induced by an arbitrary t.d.s. $(X,T)$.

\begin{lemma}\label{l:yfixed}
	Let $(Y,\sigma)$ be as (\ref{e:Y}) in Section \ref{s:proximal}.
	Then for any $x\neq *^\Z\in Y$, $(x,*^\Z)$ is not a {two-sided asymptotic} pair.
\end{lemma}

\begin{proof}
	Choose $m\in\Z$ such that $x_m\neq*$.
	By Lemma \ref{l:r_kofx}, there is a unique sequence $\se rk0$ such that $0\le r_k<p_k$, $x\in \sigma^{-r_k}Y_k$ and $r_{k+1}\equiv r_{k}\mod p_k$.
	
	Fix any $k\ge0 $.
	By (\ref{e:non*}), $m\in r_{k}+p_{k}\Z+p_{k-1}B_{k}+[0,p_{k-1})\cap\Z$.
	Let $m\in r_{k}+m_{k}p_{k}+i_kp_{k-1}+[0,p_{k-1})\cap\Z$ for some $m_{k}\in\Z$ and $i_k\in B_k$.
	Since $x_m\neq *$, by (ii) of Lemma \ref{l:forproximal}, we have $x_{r_{k}+m_{k}p_{k}+i_kp_{k-1}}\neq *$.
	Since $\sigma^{r_{k}}x\in Y_{k}$, $x_{r_{k}+m_{k}p_{k}+ip_{k-1}}\neq *$ for all $i\in B_{k}$.
	Since $\#B_{k}\to \infty$ as $k\to\infty$, $(x,*^\Z)$ is not a {two-sided asymptotic} pair.
\end{proof}

By Lemma \ref{l:r_kofx}, for any $x\in Y\setminus \{*^\Z\}$, there is a unique sequence $\se rk0$ with $0\le r_k<p_k$ such that $x\in \sigma^{-r_k}Y_k$ and $r_{k+1}\equiv r_{k}\mod p_k$ for all $k\ge 0$.
So we can define maps $r_k: Y\setminus \{*^\Z\}\to[0,p_k)\cap \Z$, $k\ge 0$ where $r_k(x)$ is as in Lemma \ref{l:r_kofx} for all $k\ge 0$.

\begin{lemma}\label{l:rxneqry}
	Let $(Y,\sigma)$ be as (\ref{e:Y}) in Section \ref{s:proximal}.
	Then for any $x\neq y\in Y\setminus\{*^\Z\}$, if there is some $K\ge 0$ such that $r_K(x)\neq r_K(y)$, then $(x,y)$ is not a {two-sided asymptotic} pair.
\end{lemma}

\begin{proof}
	Fix $x=(x_i)_{i\in\Z}\neq y=(y_i)_{i\in\Z}$ with $r_K(x)\neq r_K(y)$ for some $K\ge 0$.
	By Lemma \ref{l:r_kofx}, $r_{k+1}(x)\equiv r_{k}(x)\mod p_k$ for all $k\ge 0$, which is same for $y$.
	So we have $r_{k}(x)\neq r_{k}(y)$ for all $k\ge K$.
	
	Fix $k\ge K$.
	Without loss of generality, we assume that $0\le r_k(x)<r_k(y)<p_k$.
	We will show that
	\begin{equation}\label{e:xneqy}
	\#\{n\in\Z: d_{X_*}(x_n,y_n)=d_{X_*}(X,*)\}\ge \sqrt{\frac{n_k}2}-1.
	\end{equation}
	Set $r_k(y)-r_k(x)=ap_{k-1}+b$ where $0\le a<n_k$ and $0< b\le p_{k-1}$.
	We divide $a$ into two cases.
	\begin{itemize}
		\item {\bf Case 1:} if $0\le a\le \frac{n_k}2$, choose $m\in\Z$ with $x_m\neq *$.
		By (\ref{e:non*}), we have $m\in r_k(x)+m_kp_k+i_kp_{k-1}+[0,p_{k-1})\cap \Z$ for some $m_k\in\Z$ and $i_k\in B_k$.
		By (ii) of Lemma \ref{l:forproximal}, we have $x_{r_k(x)+m_kp_k+i_kp_{k-1}}\neq *$.
		Since $\sigma^{r_{k}(x)+m_kp_k}x\in Q_k$ and $i_k\in B_k$, we have \begin{equation}\label{e:xneq*}
		x_{r_k(x)+m_kp_k+ip_{k-1}}\neq *\text{ for all }i\in B_k.
		\end{equation}
		Then for any $i\in (A_k+a+1)\cap B_k$,
		\[
		\begin{aligned}
		&r_k(x)+m_kp_k+ip_{k-1}\\
		=&r_k(y)-ap_{k-1}-b+m_kp_k+ip_{k-1}\\
		=&r_k(y)+m_kp_k+(i-a-1)p_{k-1}+p_{k-1}-b.
		\end{aligned}
		\]
		By $i-a-1\in A_k$, $p_{k-1}-b\in[0,p_{k-1})$ and (\ref{e:proximal}), $y_{r_k(x)+m_kp_k+ip_{k-1}}=*$.
		So by (iii) of Lemma \ref{l:AB} and (\ref{e:xneq*}),
		\[
		\#\{n\in\Z: d_{X_*}(x_n,y_n)=d_{X_*}(X,*)\}\ge\#((A_k+a+1)\cap B_k)\ge \sqrt{\frac{n_k}2}-1.
		\]
		
		\item {\bf Case 2:} if $\frac{n_k}2<a< n_k$, choose $m'\in\Z$ with $y_{m'}\neq *$.
		By (\ref{e:non*}), we have $m'\in r_k(y)+m'_kp_k+i'_kp_{k-1}+[0,p_{k-1})\cap \Z$ for some $m'_k\in\Z$ and $i'_k\in B_k$.
		By (ii) of Lemma \ref{l:forproximal}, we have $y_{r_k(y)+m'_kp_k+i'_kp_{k-1}}\neq *$.
		Since $\sigma^{r_{k}(y)+m'_kp_k}y\in Q_k$ and $i'_k\in B_k$, we have
		\begin{equation}\label{e:yneq*}
		y_{r_k(y)+m'_kp_k+ip_{k-1}}\neq *\text{ for all }i\in B_k.
		\end{equation}
		Then if $0<b<p_{k-1}$, for $i\in (n_k-a+A_k)\cap B_k$,
		\[
		\begin{aligned}
		&r_k(y)+m'_kp_k+ip_{k-1}\\
		=&r_k(x)+ap_{k-1}+b+m'_kp_k+ip_{k-1}\\
		=&r_k(x)+(m'_k+1)p_k+(i-(n_k-a))p_{k-1}+b.
		\end{aligned}
		\]
		By $i-(n_k-a)\in A_k$, $b\in[0,p_{k-1})$ and (\ref{e:proximal}), $x_{r_k(y)+m'_kp_k+ip_{k-1}}=*$.
		So by (iii) of Lemma \ref{l:AB} and (\ref{e:yneq*}),
		\[
		\#\{n\in\Z: d_{X_*}(x_n,y_n)=d_{X_*}(X,*)\}\ge\#((A_k+(n_k-a))\cap B_k)\ge \sqrt{\frac{n_k}2}-1.
		\]
		
		If $b=p_{k-1}$, then $a<n_k-1$ by $r_k(y)=r_k(x)+ap_{k-1}+b$.
		So for $i\in (n_k-a-1+A_k)\cap B_k$,
		\[
		r_k(y)+m'_kp_k+ip_{k-1}=r_k(x)+(m'_k+1)p_k+(i-(n_k-a-1))p_{k-1}.
		\]
		Since $i-(n_k-a-1)\in A_k$ and (\ref{e:proximal}), we have  $x_{r_k(y)+m'_kp_k+ip_{k-1}}=*$.
		So by (iii) of Lemma \ref{l:AB} and (\ref{e:yneq*}),
		\[
		\#\{n\in\Z: d_{X_*}(x_n,y_n)=d_{X_*}(X,*)\}\ge\#((A_k+(n_k-a-1))\cap B_k)\ge \sqrt{\frac{n_k}2}-1.
		\]
	\end{itemize}

	To sum up, (\ref{e:xneqy}) is proved.
	Since $n_k\to\infty$ as $k\to\infty$, it ends the proof.
\end{proof}

Finally, for the last case of $x\neq y$ with the same sequence $\se rk0$, we will show more about points in $Y$.
Recall that $C_k=B_k\setminus\{0,n_k-1\}$, $f_k|_{C_k}: C_k\to f_k(B_k)$ is a bijection and $f_k^{-1}: f_k(B_k)\to C_k$ is the inverse of the bijection $f_k|_{C_k}$.
\begin{lemma}\label{l:sameasx_0}
	Let $(Y,\sigma)$ be as (\ref{e:Y}) in Section \ref{s:proximal}, and $k\ge 0$.
	For any $x=(x_i)_{i\in\Z},y=(y_i)_{i\in\Z}\in Y_k$, we have
	\begin{itemize}
		\item[(i)] if $x_0=y_0$, then $x|_{[0,p_k)}=y|_{[0,p_k)}$;
		\item[(ii)] for any $i\in B_k$,  $(\sigma^{ip_{k-1}}x)|_{[0,p_{k-1})}=(\sigma^{f^{-1}_k(f_k(i))p_{k-1}})|_{[0,p_{k-1})}$.
	\end{itemize}
\end{lemma}

\begin{proof}
	We prove (i) by induction.
	For $k=0$, let $x,y\in Y_0$.
	So we have $x_i=*=y_i$ for $i\in A_0$, and $x_i=T_*^{f_0(i)}x_0$ and $y_i=T_*^{f_0(i)}y_0$ for $i\in B_0$.
	So (i) holds for $k=0$.
	
	Suppose that (i) holds for some $k\ge 0$.
	For any $x,y\in Y_{k+1}$ with $x_0=y_0$, we have $x_{ip_k}=*=y_{ip_k}$ for $i\in A_{k+1}$, and $x_{ip_k}=T_*^{b'_kf_{k+1}(i)}x_0$ and $y_{ip_k}=T_*^{b'_kf_{k+1}(i)}y_0$ for $i\in B_{k+1}$.
	So $x_{ip_{k}}=y_{ip_{k}}$ for $0\le i<n_{k+1}$.
	Since $\sigma^{ip_{k}}x,\sigma^{ip_{k}}y\in Y_{k}$, by the induction hypothesis, we have $x|_{[ip_{k},(i+1)p_{k})}=y|_{[ip_{k},(i+1)p_{k})}$ for $0\le i<n_{k+1}$.
	So (i) holds for $k+1$.
	
	For any $x\in Y_{k}\subset Q_{k}$, noticing that $f_k\circ f_k^{-1}=Id_{f_k(B_k)}$ is the identity of $f_k(B_k)$, we have $x_{f_k^{-1}(f_k(i))p_{k-1}}=T^{b'_{k-1}f_k(i)}x_0=x_{ip_{k-1}}$ for $i\in B_k$.
	Then by (i), (ii) is proved.
\end{proof}

To prove the last case, we also need the following t.d.s.

\begin{theorem}\label{t:nonasym}
	\cite[Example 3.4]{LS99}
	There is a t.d.s. $(X,T)$ such that $h_{\rm top}(X,T)>0$ and all the proper pairs are not {two-sided asymptotic} pairs.
\end{theorem}

If we take $(X,T)$ be as in Theorem \ref{t:nonasym} and construct $(Y,\sigma)$ as in Section \ref{s:proximal}, then we can show that all the proper pairs of $(Y,\sigma)$ are not {two-sided asymptotic} pairs.

\begin{proposition}\label{p:nonasym}
	Let $(X,T)$ be a t.d.s.  satisfying that all  proper pairs of $(X,T)$ are not {two-sided asymptotic} pairs, and define $(Y,\sigma)$ as (\ref{e:Y}) in Section \ref{s:proximal}.
	Then for any $x\neq y\in Y$, $(x,y)$ is not {two-sided asymptotic}.
\end{proposition}

\begin{proof}
	Fix $x=(x_i)_{i\in\Z}\neq y=(y_i)_{i\in\Z}\in Y$.
	By Lemma \ref{l:yfixed} and Lemma \ref{l:rxneqry}, we can assume that $x,y\in Y\setminus\{*^\Z\}$ and $x,y\in\bigcap_{k=0}\sigma^{-r_k}Y_k$ for the same sequence $\se rk0$ with $0\le r_k<p_k$ and $r_{k+1}\equiv r_{k}\mod p_k$ for all $k\ge 0$.
	Without loss of generality, we assume that $x_m\neq *$ and $x_m\neq y_m$ for some $m\in\Z$.
	
	By (\ref{e:non*}), $m\in r_k+p_k\Z+p_{k-1}B_{k}+[0,p_{k-1})\cap\Z$ for all $k\ge 0$.
	For each $k\ge 0$, there are unique $m_k\in\Z$ and unique $i_k\in B_k$ such that
	\[
	m\in r_k+m_kp_k+i_kp_{k-1}+[0,p_{k-1})\cap\Z\subset r_k+m_kp_k+[0,p_{k})\cap\Z.
	\]
	The integer $m_k$ is also unique such that $m\in r_k+m_kp_k+[0,p_{k})\cap\Z$.
	
	Next, we show by induction that for all $k\ge 0$, \begin{equation}\label{e:m}
	m=r_k+m_kp_k+\sum_{k'=0}^{k}i_{k'}p_{k'-1}.
	\end{equation}
	For $k=0$, $m=r_0+m_0p_0+i_0$ by $p_{-1}=1$.
	Suppose that (\ref{e:m}) holds for some $k\ge 0$.
	Since
	\[
		\begin{aligned}
		m\in & r_{k+1}+m_{k+1}p_{k+1}+i_{k+1}p_{k}+[0,p_{k})\cap\Z\\
		=&r_k+\left(m_{k+1}n_{k+1}+\frac{r_{k+1}-r_{k}}{p_{k}}+i_{k+1}\right)p_{k}+[0,p_{k})\cap\Z
		\end{aligned}
	\]
	where $\frac{r_{k+1}-r_{k}}{p_{k}}\in\Z$ by $r_{k+1}\equiv r_k\mod p_k$,
	we have
	\[
	r_{k+1}+m_{k+1}p_{k+1}+i_{k+1}p_{k}=r_k+m_kp_k
	\]
	by the uniqueness of $m_k$.
	So by the induction hypothesis,
	\[
		m=r_k+m_kp_k+\sum_{k'=0}^{k}i_{k'}p_{k'-1}=r_{k+1}+m_{k+1}p_{k+1}+\sum_{k'=0}^{k+1}i_{k'}p_{k'-1},
	\]
	that is, (\ref{e:m}) holds for $k+1$.
	
	Let $S=\{k\ge0:\tau_k(i_k)\in\{0,1,b_k-2,b_k-1\}\}$.
	The rest of proof is divided into two cases.
	
	{\bf Case 1:} assume that $\#S=\infty$.
	In this case, we will show that $x_m$ and $y_m$ appear infinitely many times at the same coordinate of $x$ and $y$, respectively.
	Notice that for $i\in \tau_k^{-1}\{0,1,b_k-2,b_k-1\}$, $f_k(i)$ has two $f_k$-preimages.
	For $k\in S$, let $i'_k\neq i_k$ be a $f_k$-preimage of $f_k(i_k)$.
	By (ii) of Lemma \ref{l:sameasx_0},
	\[
	(\sigma^{r_k+m_kp_k+i_kp_{k-1}}x)|_{[0,p_{k-1})}=(\sigma^{r_k+m_kp_k+i'_kp_{k-1}}x)|_{[0,p_{k-1})}.
	\]
	In particular, $x_m=x_{m^{(k)}}$ where $m^{(k)}=m+(i'_k-i_k)p_{k-1}$.
	Similarly, $y_m=y_{m^{(k)}}$.
	Since $\#S=\infty$ and $|m^{(k)}-m|=|i'_k-i_k|p_{k-1}\ge p_{k-1}\to\infty$ as $k\to\infty$, $(x,y)$ is not {two-sided asymptotic}.
	
	{\bf Case 2:}
	assume that $\#S<\infty$.
	In this case, we will show that $x$ and $y$ contain the $\Z$-orbit of $x_m$ and $y_m$, respectively.
	First, we find another integer $m'$ instead of $m$.
	For $k\in S$, since $f_k|_{C_k}: C_k\to B_k$ is a bijection, there is $i'_k\in C_k$ such that $f_k(i'_k)=f_k(i_k)$.
	For $k\notin S$, set $i'_k=i_k\in C_k\setminus\{\tau_k^{-1}(1),\tau_k^{-1}(b_k-2)\}$.
	
	Let
	\[
		m':=m+\sum_{k\in S}(i'_k-i_k)p_{k-1}=m+\sum_{k=0}^{\infty}(i'_k-i_k)p_{k-1}.
	\]
	So for $K\ge 0$,
	\[
	\begin{aligned}
	m'=&m+\sum_{k=0}^{\infty}(i'_k-i_k)p_{k-1}\\
	=&r_K+m_Kp_K+\sum_{k=K+1}^{\infty}(i'_k-i_k)p_{k-1}+\sum_{k=0}^{K}i'_kp_{k-1}\quad\text{(by (\ref{e:m}))}\\
	\in & r_K+p_K\Z+i'_Kp_{K-1}+[0,p_{K-1})\cap\Z
	\end{aligned}
	\]
	where the last statement holds since $i'_k<n_k$ for all $k$ and then the sum $\sum_{k=0}^{K-1}i'_kp_{k-1}<p_{K-1}$.
	For $K\ge 0$, set
	\begin{equation}\label{e:m'_k}
	m'_K=m_K+\sum_{k=K+1}^{\infty}(i'_k-i_k)\frac{p_{k-1}}{p_{K}}.
	\end{equation}
	So
	\begin{equation}
		m'=r_K+m'_Kp_K+\sum_{k=0}^{K}i'_kp_{k-1}\in r_K+m'_Kp_K+i'_Kp_{K-1}+[0,p_{K-1})\cap\Z.
	\end{equation}
	
	We claim that $x_m=x_{m'}$ and $y_{m}=y_{m'}$.
	
	Let $\hat{K}=\max S$ and $x^{(0)}=\sigma^{r_{\hat{K}+1}+m_{\hat{K}+1}p_{\hat{K}+1}+i_{\hat{K}+1}p_{\hat{K}}}x\in Y_{\hat{K}}$.
	So $i_{k}=i'_k$ for $k>\hat{K}$.
	Since $x^{(0)}\in Y_{\hat{K}}$, by (ii) of Lemma \ref{l:sameasx_0},
	\begin{equation}\label{e:xm=xm'1}
	(\sigma^{i_{\hat{K}}p_{\hat{K}-1}}x^{(0)})|_{[0,p_{\hat{K}-1})}=(\sigma^{i'_{\hat{K}}p_{\hat{K}-1}}x^{(0)})|_{[0,p_{\hat{K}-1})}.
	\end{equation}
	Let $x^{(1)}=\sigma^{i'_{\hat{K}}p_{\hat{K}-1}}x^{(0)}\in Y_{\hat{K}-1}$.
	Again, by (ii) of Lemma \ref{l:sameasx_0},
	\begin{equation}\label{e:xm=xm'2}
	(\sigma^{i_{\hat{K}-1}p_{\hat{K}-2}}x^{(1)})|_{[0,p_{\hat{K}-2})}=(\sigma^{i'_{\hat{K}-1}p_{\hat{K}-2}}x^{(1)})|_{[0,p_{\hat{K}-2})}.
	\end{equation}
	So
	\[
	\begin{aligned}
	&(\sigma^{i_{\hat{K}}p_{\hat{K}-1}+i_{\hat{K}-1}p_{\hat{K}-2}}x^{(0)})|_{[0,p_{\hat{K}-2})}\\
	=&(\sigma^{i_{\hat{K}}p_{\hat{K}-1}}x^{(0)})|_{[i_{\hat{K}-1}p_{\hat{K}-2},(i_{\hat{K}-1}+1)p_{\hat{K}-2})}\\
	=&x^{(1)}|_{[i_{\hat{K}-1}p_{\hat{K}-2},(i_{\hat{K}-1}+1)p_{\hat{K}-2})}\quad\text{(by (\ref{e:xm=xm'1}))}\\
	=&(\sigma^{i_{\hat{K}-1}p_{\hat{K}-2}}x^{(1)})|_{[0,p_{\hat{K}-2})}\\
	=&(\sigma^{i'_{\hat{K}}p_{\hat{K}-1}+i'_{\hat{K}-1}p_{\hat{K}-2}}x^{(0)})|_{[0,p_{\hat{K}-2})}\quad\text{(by (\ref{e:xm=xm'2})).}
	\end{aligned}
	\]
	
	Similarly, for $1\le k\le \hat{K}$, we have $x^{(k)}=\sigma^{i'_{\hat{K}-k+1}p_{\hat{K}-k}}x^{(k-1)}\in Y_{\hat{K}-k}$ and
	\[
	(\sigma^{\sum_{k'=\hat{K}-k}^{\hat{K}}i_{k'}p_{k'-1}}x^{(k)})|_{[0,p_{\hat{K}-k-1})}=(\sigma^{\sum_{k'=\hat{K}-k}^{\hat{K}}i'_{k'}p_{k'-1}}x^{(k)})|_{[0,p_{\hat{K}-k-1})}.
	\]
	
	Therefore, we have $(\sigma^{\sum_{k'=0}^{\hat{K}}i_{k'}p_{k'-1}}x^{(0)})_0=(\sigma^{\sum_{k'=0}^{\hat{K}}i'_{k'}p_{k'-1}}x^{(0)})_0$, that is, $x_m=x_{m'}$.
	Replacing $x$ by $y$, we also have $y_{m}=y_{m'}$ by the same arguments.
	
	Now we have found $m'$ instead of $m$ such that $x_{m'}\neq *$, $x_{m'}\neq y_{m'}$ and for all $k\ge 0$,
	$m'\in r_{k}+m'_{k}p_{k}+i'_{k}p_{k-1}+[0,p_{k-1})$ where $i'_{k}\in C_k$.
	So we have $m'-(r_k+m'_kp_k)\in J_k$ for all $k\ge 0$, noticing that for any $k\ge k'\ge0$,
	\[
	m'-(r_k+m'_kp_k)=\sum_{k''=0}^{k}i'_{k''}p_{k''-1}\in p_{k'}\Z+i'_{k'}p_{k'-1}+[0,p_{k'-1})\cap\Z.
	\]
	
	Finally, we will show that points $x$ and $y$ contain the $\Z$-orbit of $x_{m'}$ and $y_{m'}$ at the same coordinate, respectively.
	To show this similarity, fix $z=(z_i)_{i\in\Z}\in\{x,y\}$.
	
	For $k\ge 0$, let $\alpha_k=\phi_k^{-1}(m'-(r_k+m'_kp_k))$.
	Since $\sigma^{r_k+m'_kp_k}z\in Y_k$, by (i) of Lemma \ref{l:phi_k}, we have $z_{r_k+m'_kp_k+\phi_k(n)}=T_*^nz_{r_k+m'_kp_k}$ for all $n\in I_k$.
	In particular,
	\[
	z_{m'}=z_{r_k+m'_kp_k+\phi_k(\alpha_k)}=T_*^{\alpha_k}z_{r_k+m'_kp_k}
	\]
	So for all $k\ge 0$,
	\begin{equation}\label{e:orbit}
	z_{r_k+m'_kp_k+\phi_k(n+\alpha_k)}=T_*^{n}z_{m'}\text{ for all }n\in I_k-\alpha_k.
	\end{equation}
	
	We will show that for any integer $N\ge 1$, there is $k\ge 0$ such that
	\begin{equation}\label{e:Zorbit}
	I_k-\alpha_k\supset (-N,N)\cap\Z.
	\end{equation}
	Fix integer $N\ge 1$.
	Choose $k>\max S$ such that $\#I_{k-1}>N$.
	Since $k>\max S$, we have $i'_k=i_k\in C_k\setminus\{\tau_k^{-1}(1),\tau_k^{-1}(b_k-2)\}$.
	So $\tau_{k}(i'_k)\in \{2,3,\dots,b_k-3\}$ and then $f_k(i'_k)\in\{-(b_k-4),\dots,-2,-1\}$.
	By (ii) of Lemma \ref{l:phi_k},
	$\alpha_k\in f_k(i'_k)b'_{k-1}+I_{k-1}$ since $m'-(r_k+m'_kp_k)\in i'_kp_{k-1}+J_{k-1}$.
	Since $I_k=f_k(B_k)b'_{k-1}+I_{k-1}$ and $f_k(B_k)=\{-(b_k-3),\dots,-2,-1,0\}$, we have
	\[
		I_k-\alpha_k\supset I_{k-1}\cup(b'_{k-1}+I_{k-1})=(-b'_{k-1},b'_{k-1}]\cap\Z.
	\]
	And by $b'_{k-1}=\#I_{k-1}>N$, (\ref{e:Zorbit}) is proved.
	
	Now we have proved that for any $z\in \{x,y\}$, (\ref{e:orbit}) holds, and for any $N\ge 1$, there is $k\ge 0$ such that (\ref{e:Zorbit}) holds.
	Therefore if $y_{m'}=*$, $(x,y)$ is not {two-sided asymptotic} by $x_{m'}\neq *$.
	In the case of $y_{m'}\neq *$, since all the proper pairs of $(X,T)$ are not {two-sided asymptotic}, so is $(x,y)$.
\end{proof}

By Proposition \ref{p:nonasym}, the existence of homeomorphisms with positive topological entropy such that all the proper pairs of $(X,T)$ are not {two-sided asymptotic} implies the existence of completely  Li--Yorke chaotic homeomorphisms with positive topological entropy.
So by Theorem \ref{t:nonasym}, we are ready to prove Theorem \ref{t:main}.

\begin{proof}[Proof of Theorem \ref{t:main}]
	By Theorem \ref{t:nonasym}, choose a t.d.s. $(X,T)$  such that $h_{\rm top}(X,T)>0$ and all the proper pairs of $X$ are not {two-sided asymptotic}.
	And then choose a decreasing sequence $\se{\epsilon}k0$ such that $0<\epsilon_k<1/2$ for $k\ge 0$, and
	\[
		\prod_{k=0}^{\infty}(1-2\epsilon_k)>0.
	\]
	Let $(Y,\sigma)$ be defined as (\ref{e:Y}) in Section \ref{s:proximal}.
	By Proposition \ref{p:entropy}, $h_{\rm top}(Y,\sigma)>0$.
	By Proposition \ref{p:proximal} and Proposition \ref{p:nonasym}, all the proper pairs are either positively or negatively scrambled pairs, which ends the proof.
\end{proof}

\section*{Acknowledgements}
We would like to thank Professor Guohua Zhang for many useful comments, and Xulei Wang for pointing out some typos in the earlier manuscript.
This research is supported by NNSF of China (Nos. 12371196 and 12031019) and China Post-doctoral Grant (Nos. BX20230347 and 2023M743395).

\end{document}